\documentclass[11pt, english]{article}
\usepackage[margin=3cm]{geometry}
\usepackage{amsbsy,amsmath,amsthm,amssymb}

\usepackage{fourier}

\usepackage{xcolor,bm}
\usepackage{esint}
\usepackage{babel}
\usepackage[colorlinks=true, citecolor = blue]{hyperref}

\usepackage{indentfirst}
\usepackage{changepage}
\usepackage{setspace}
\usepackage{indentfirst}
\usepackage{graphicx}
\usepackage{calc}

\allowdisplaybreaks

\theoremstyle{plain}
\newtheorem{theorem}{Theorem}[section]

\newtheorem{corollary}[theorem]{Corollary}

\newtheorem{remark}[theorem]{Remark}

\theoremstyle{definition}

\newenvironment{proof of theorem 1.1}{{\noindent \em Proof of Theorem 1.1.}}{\hfill $\Box$\par}
\newenvironment{proof of theorem 1.2}{{\noindent \em Proof of Theorem 1.2.}}{\hfill $\Box$\par}

\DeclareSymbolFont{EulerExtension}{U}{euex}{m}{n}
\DeclareMathSymbol{\euintop}{\mathop} {EulerExtension}{"52}
\DeclareMathSymbol{\euointop}{\mathop} {EulerExtension}{"48}

\numberwithin{equation}{section}

\newcommand{\CC}{\ensuremath{\mathbb{C}}}

\newcommand{\ZZ}{\ensuremath{\mathbb{Z}}}

\newcommand{\DD}{\ensuremath{\mathbb{D}}}
\newcommand{\md}{\mathrm{d}}
\newcommand{\mo}{\mathcal{O}}
\newcommand{\mi}{\mathrm{i}}
\newcommand{\mf}{\mathfrak{f}}
\newcommand{\bU}{\textbf{U}}

\begin{document}
	
\title{Asymptotics of Saran's hypergeometric function $F_K$}
\author{Peng-Cheng Hang$^{\rm 1}$  and Min-Jie Luo$^{\rm 2}$\thanks{Corresponding author}}
\date{}
\maketitle
\begin{center}\small
	$^{1}$\emph{Department of Mathematics, School of Mathematics and Statistics,\\ Donghua University, Shanghai 201620, \\ 
			People's Republic of China.}\\
	E-mail: \texttt{mathroc618@outlook.com}
	\vspace{0.3cm}
	
	$^{2}$\emph{Department of Mathematics, School of Mathematics and Statistics,\\ Donghua University, Shanghai 201620, \\ 
		People's Republic of China.}\\
	E-mail: \texttt{mathwinnie@live.com}, \texttt{mathwinnie@dhu.edu.cn}
\end{center}
	
\vspace{0.5cm}
	
\begin{abstract}
	In this paper, we first establish asymptotic expansions of the Humbert function $\Psi_1$ for one large variable. The resulting expansions are then used to derive an asymptotic expansion of Saran's hypergeometric function $F_K$ when two of its variables become simultaneously large.
	\\
		
		
\end{abstract}

\section{Introduction}\label{Section 1}

Saran's hypergeometric function $F_K$ is defined by \cite[p. 294, Eq. (2.4)]{Saran 1955}
\begin{equation}\label{F_K def}
	F_K[\alpha_1,\alpha_2,\alpha_2,\beta_1,\beta_2,\beta_1;\gamma_1,\gamma_2,\gamma_3;x,y,z]=\sum_{m,n,p=0}^{\infty}\frac{(\alpha_1)_{m}(\alpha_2)_{n+p}(\beta_1)_{m+p}(\beta_2)_n}{(\gamma_1)_m(\gamma_2)_n(\gamma_3)_p}\frac{x^m}{m!}\frac{y^n}{n!}\frac{z^p}{p!},
\end{equation}
where $(x,y,z)\in\DD_K:=\left\{(x,y,z)\in\CC^3:|x|<1,|y|<1,|z|<(1-|x|)(1-|y|)\right\}$. It has been shown in \cite[p. 1]{Luo-Xu-Raina-2022} that $\mathbb{D}_K$ is a complete Reinhardt domain. Using the series manipulation technique, we can obtain the following equivalent forms of \eqref{F_K def}:
\begin{align}
	&F_K[\alpha_1,\alpha_2,\alpha_2,\beta_1,\beta_2,\beta_1;\gamma_1,\gamma_2,\gamma_3;x,y,z]\notag\\
	&\hspace{1cm}=\sum_{p=0}^{\infty}\frac{(\alpha_2)_p(\beta_1)_p}{p!(\gamma_3)_p}
	{}_{2}F_{1}\left[\begin{matrix}
		\beta_1+p,\alpha_1\\
		\gamma_1
	\end{matrix}; x\right]
	{}_{2}F_{1}\left[\begin{matrix}
		\alpha_2+p,\beta_2\\
		\gamma_2
	\end{matrix}; y\right]z^p\notag\\
	&\hspace{1cm}=\sum_{m=0}^{\infty}\frac{(\alpha_1)_m(\beta_1)_m}{m!(\gamma_1)_m}F_2[\alpha_2,\beta_2,\beta_1+m;\gamma_2,\gamma_3;y,z]x^m,\label{FK-AppellF2}
\end{align}
where ${}_{2}F_{1}$ denotes the familiar Gauss hypergeometric function and $F_2$ denotes the second Appell hypergeometric function (see \cite[p. 23, Eq. (3)]{Srivastava-Karlsson}). Over the past few years, more and more attention has been given to various properties and applications of this function (see \cite{Antonova-Dmytryshyn-Goran-2023},  \cite{Dmytryshyn-Goran-2024}, \cite{Luo-Raina 2021} and \cite{Luo-Xu-Raina-2022} and the references therein). 

As a special function in several complex variables, the asymptotic behaviour of $F_K$ is rather complicated. Recently, Luo and Raina \cite{Luo-Raina 2021} obtain the complete asymptotic formulas for $F_K$ when one of its variables is large. So it is natural to study the asymptotic behaviour of $F_K$ when two variables are large. In fact, as a special case, Garcia and L\'{o}pez \cite{Garcia-Lopez 2010} have deduced the behaviour of $F_2$ when $x,y\to \infty$ with $|x|/|y|$ bounded. Their approach is based on the asymptotic method introduced in \cite{Lopez 2008} and a suitable Laplace integral representation for $F_2$. In view of the expression \eqref{FK-AppellF2}, where $y$ and $z$ appear together as variables of $F_2$, we shall study the asymptotic behaviour of $F_K$ when $y,z\rightarrow\infty$ with $|y|/|z|$ bounded. Note also that since
\[
	F_K[\alpha_1,\alpha_2,\alpha_2,\beta_1,\beta_2,\beta_1;\gamma_1,\gamma_2,\gamma_3;x,y,z]=F_K[\beta_2,\beta_1,\beta_1,\alpha_2,\alpha_1,\alpha_2;\gamma_2,\gamma_1,\gamma_3;y,x,z],
\]
the behaviour for large $x,z$ is equivalent to that for large $y,z$. 

To accomplish this, we need to first study the asymptotics of the Humbert function $\Psi_1$ defined by \cite[p. 26, Eq. (21)]{Srivastava-Karlsson}
\begin{equation}\label{Psi_1 def}
	\Psi_1[a,b;c,c';x,y]:=\sum_{m,n=0}^{\infty}\frac{\left(a\right)_{m+n}\left(b\right)_m}{\left(c\right)_m\left(c'\right)_n}\frac{x^m}{m!}\frac{y^n}{n!}
	~(|x|<1,|y|<\infty).
\end{equation}
The motivation of considering $\Psi_1$ is that $F_K$ possesses the following Laplace integral representation \cite[p. 134, Eq. (3)]{Saran 1957}:
\begin{align}\label{F_K Laplace integral}
		&F_K[\alpha_1,\alpha_2,\alpha_2,\beta_1,\beta_2,\beta_1;\gamma_1,\gamma_2,\gamma_3;x,y,z]\notag\\
		&\hspace{1cm}=\frac{1}{\Gamma(\alpha_2)}\int_0^{\infty}\mathrm{e}^{-s}s^{\alpha_2-1}
		{}_1F_1\left[\begin{matrix}
			\beta_2\\
			\gamma_2
		\end{matrix};ys\right]
		\Psi_1[\beta_1,\alpha_1;\gamma_1,\gamma_3;x,zs]\md s,
\end{align}
where $\Psi_1$ is given by \eqref{Psi_1 def} and ${}_{1}F_{1}$ is Kummer's confluent hypergeometric function. The convergence condition of \eqref{F_K Laplace integral}
\[
	\Re(\alpha_2)>0,~\Re((1-x)(1-y))>\Re(z),
\]
given in Saran's paper, is incorrect for it is derived from an incomplete estimate about $\Psi_1$ (see \cite[p. 143, Eq. (3)]{Saran 1957}). To clarify the convergence condition, we must derive the correct estimate of $\Psi_1$.

The paper is organized as follows:

In Section \ref{Section 2}, we provide some necessary information about the Gauss hypergeometric function ${}_{2}F_{1}$, the Humbert function $\Psi_1$ and the Kamp\'e de F\'eriet function $F_{0:1;2}^{1:1;2}$. 

Section \ref{Section 3} devotes to the asymptotics of $\Psi_1$. We first establish a contour integral representation for $\Psi_1$, from which an  asymptotic expansion of $\Psi_1$ for large $y$ is obtained in the left half-plane. Then, by making use of a result of Olver \cite{Olver 1970}, we obtain under certain restrictions the complete asymptotic expansions of $\Psi_1$ for large $y$ in the right half-plane. Finally, we obtain the asymptotic expansion of $\Psi_1$ for large $x$. 

Section \ref{Section 4} offers a comprehensive proof of the main result. First we convert the integral representation \eqref{F_K Laplace integral} into the form of the Mellin convolution. The desired approximation of $F_K$ is then obtained by applying L\'{o}pez's theorem \cite{Lopez 2008} to that Mellin convolution integral.

In Appendix \ref{Appendix A}, we compute an important Mellin transform of a product of ${}_{1}F_{1}$ and $\Psi_1$. Finding different proofs of it may have independent interest.

\section{Preliminaries}\label{Section 2}
\subsection{Asymptotics of the Gauss hypergeometric function $_2F_1$}
Wagner \cite[p. 2, Satz 1]{Wagner 1982} proved that
\begin{equation}\label{Wager formula}
	\begin{split}
		\frac{1}{\Gamma(c)}{}_{2}F_{1}\left[\begin{matrix}
			a,b\\
			c
		\end{matrix};z\right]&= \frac{\left(-az\right)^{-b}}{\Gamma(c-b)}[1+o(1)] \\
		&\hspace{0.5cm}+\frac{\mathrm{e}^{\pm \pi i(b-c)}}{\Gamma(b)}\left(1-z\right)^{c-a-b}\left(-az\right)^{b-c}[1+o(1)],\ \text{as }|a|\to\infty
	\end{split}
\end{equation}
with $b$ and $c$ ($c\notin \mathbb{Z}_{\leq0}$) fixed complex numbers, $z\ne 0$ and $|\arg(1-z)|<\pi$. The choice of the sign $\pm$ in the exponent and the determination of the range of $\arg(-az)$ depend on the conditions satisfied by the parameters. For convenience, we present some corollaries of Wagner's formula \eqref{Wager formula} which we shall use in the sequel. Similar analysis of \cite[Eq. (12)]{Luo-Raina 2021} gives that for $n\in\ZZ_{>0}$,
\begin{equation}\label{2F1 integer parameter}
	\left|{}_2F_1\left[\begin{matrix}
		a+n,b \\
		c
	\end{matrix};z\right]\right|=\mo\left(n^{-\omega}\rho_z^{n}\right),\ \ \text{as }n\to \infty,
\end{equation}
where
\begin{equation}\label{omega-rho def}
	\omega:=\min\{\Re(b),\Re(c-b)\}\ \text{ and }\ \rho_z:=\max\big\{1,\left|1-z\right|^{-1}\big\}.
\end{equation}
We use Landau's notation $f=\mathcal{O}(g)$ to mean that $|f|\leq C|g|$ for a suitable positive number $C$, which may be absolute or depend on various parameters, in which case the dependence may be indicated in a subscript.

We need the following estimate \cite[p. 4]{Luo-Raina 2021}:
\begin{equation}\label{2F1 imaginary parameter}
	\left|{}_2F_1\left[\begin{matrix}
		a+\mi t,b \\
		c
	\end{matrix};z\right]\right|=\mo\left(\left|t\right|^{-\omega}\mathrm{e}^{|t|\max\left\{0,\mp\arg(1-z)\right\}}\right),\ \ t\to \pm \infty,
\end{equation}
where $\omega$ is given in \eqref{omega-rho def}. We also need the behaviour of
\[\left|{}_2F_1\left[\begin{matrix}
	a+s,b \\
	c
\end{matrix};z\right]\right|,\ s\in\mathfrak{C},\]
where $\mathfrak{C}$ is the semi-circle $s=(N+1/2)\mathrm{e}^{\mi\theta}$ $(|\theta|\leq\pi/2)$ and $N\to\infty$ is taken through integral values. 

Note that for $s\in\mathfrak{C}$,
\[
\left|\left(1-z\right)^{-s}\right|=\mathrm{e}^{(N+1/2)(\sin\theta\arg\left(1-z\right) -\cos\theta\log|1-z|)}.
\]
This together with \eqref{Wager formula} implies that for $s\in\mathfrak{C}$,
\begin{equation}\label{2F1 semi-circle}
	\left|{}_2F_1\left[\begin{matrix}
		a+s,b \\
		c
	\end{matrix};z\right]\right|
	=\mo\left(N^{-\omega}\mathrm{e}^{(N+1/2)(|\sin\theta||\arg(1-z)| -\cos\theta\log(1-|z|)}\right),\ N\to\infty,
\end{equation}
where $|z|<1$ and $\omega$ is given in \eqref{omega-rho def}.

\subsection{Properties of the Humbert function $\Psi_1$}

We have (see \cite[p. 75]{Humbert-1992} and \cite[p. 118, Eq. (83)]{Brychkov-Saad 2014})
\begin{equation}\label{Psi_1 series representation}
	\Psi_1[a,b;c,c';x,y]=\sum_{n=0}^{\infty}\frac{\left(a\right)_n}{\left(c'\right)_n}\,_2F_1\left[\begin{matrix}
		a+n,b \\
		c
	\end{matrix};x\right]\frac{y^n}{n!},\ \ |x|<1,~|y|<\infty.
\end{equation}
Taking \eqref{2F1 integer parameter} into account, the summand in \eqref{Psi_1 series representation} has the order of magnitude
\[
\mo\left(n^{\Re(a-c')-\omega}\frac{(\rho_x|y|)^n}{n!}\right),~
n\to\infty,
\]
which implies that the series converges absolutely in the region
\[
\mathbb{D}_{\Psi_1}:=\left\{(x,y)\in\mathbb{C}^2:x\ne 1,~|\arg(1-x)|<\pi,~|y|<\infty\right\}.
\]
So the series in \eqref{Psi_1 series representation} provides an analytic continuation of $\Psi_1$ to $\mathbb{D}_{\Psi_1}$.

We also have \cite[p. 115, Eq. (38)]{Brychkov-Saad 2014}
\begin{equation}\label{Psi_1 integral representation}
	\Psi_1[a,b;c,c';x,y]=\frac{\Gamma(c)}{\Gamma(b)\Gamma(c-b)}\int_0^1 u^{b-1}\left(1-u\right)^{c-b-1}\left(1-ux\right)^{-a}
	{}_1F_1\left[\begin{matrix}
		a\\
		c'
	\end{matrix};\frac{y}{1-ux}\right]\md u,
\end{equation}
which is valid for $\Re(c)>\Re(b)>0$ and $(x,y)\in\DD_{\Psi_1}$. The integral representation \eqref{Psi_1 integral representation} follows from the series representation \eqref{Psi_1 series representation} and Euler integral representation for ${}_2F_1$ (see \cite[p. 388, Eq. (15.6.1)]{NIST Handbook}).

We shall also use the Kummer transformation of $\Psi_1$ \cite[Eq. (2.54)]{Choi-Hasanov 2011}
\begin{equation}\label{Psi_1 Kummer transformation}
	\Psi_1[a,b;c,c';x,y]=\left(1-x\right)^{-a}\Psi_1\left[a,c-b;c,c';\frac{x}{x-1},\frac{y}{1-x}\right].
\end{equation}

There are some useful identities about $\Psi_1$ in the literature (see \cite{Brychkov-Saad 2014}, \cite{Choi-Hasanov 2011}, \cite{Humbert-1992} and \cite{Hai-Yakubovich-Book}). But we still know very little about its asymptotics. Joshi and Arya \cite{Joshi-Arya-1982} established a set of bilateral inequalities for $\Psi_1$ under certain rather restrictive conditions. By using a Taubarian theorem for Laplace transform, Wald and Henkel \cite{Wald-Henkel-2018} found the leading asymptotic behaviour of the Humbert functions $\Phi_2$, $\Phi_3$ and $\Xi_2$ when the absolute values of the two independent variables become simultaneously large. They also considered $\Psi_1$ and pointed out that their technique can not handle $\Psi_1$ (see \cite[p. 99]{Wald-Henkel-2018}). Therefore, our work in Section \ref{Section 3} may help advance the research in this direction.

\subsection{Kamp\'e de F\'eriet function $F_{0:1;2}^{1:1;2}$}

We need the following special case of the Kamp\'e de F\'eriet function \cite[p. 27]{Srivastava-Karlsson}: 
\begin{equation}\label{F_{0:1;2}^{1:1;2} def}
	F_{0:1;2}^{1:1;2}\left[\begin{matrix}
		a_1:b_1;d_1,d_2\\
		-:  c_1;e_1,e_2
	\end{matrix};x,y\right]=\sum_{m,n=0}^{\infty}\frac{\left(a_1\right)_{m+n}\left(b_1\right)_m\left(d_1\right)_n\left(d_2\right)_n}{\left(c_1\right)_m\left(e_1\right)_n\left(e_2\right)_n}\frac{x^m}{m!}\frac{y^n}{n!}, ~|x|+|y|<1.
\end{equation}
In view of $(a_1)_{n+m}=(a_1)_n(a_1+n)_m$, we have from \eqref{F_{0:1;2}^{1:1;2} def} that
\begin{equation}\label{F_{0:1;2}^{1:1;2} series}
	F_{0:1;2}^{1:1;2}\left[\begin{matrix}
		a_1:b_1;d_1,d_2\\
		-:  c_1;e_1,e_2
	\end{matrix};x,y\right] =\sum_{n=0}^{\infty}\frac{\left(a_1\right)_n\left(d_1\right)_n\left(d_2\right)_n}{\left(e_1\right)_n\left(e_2\right)_n}{}_2F_1\left[\begin{matrix}
		a_1+n,b_1\\
		c_1
	\end{matrix};x\right]\frac{y^n}{n!}.
\end{equation}
Denote by $F_n$ the summand of the series \eqref{F_{0:1;2}^{1:1;2} series}. Using the estimate \eqref{2F1 integer parameter} gives
\[
\left|F_n\right|=\mo\left(n^{\Re(a_1+d_1+d_2-e_1-e_2)-\omega'-1}(\rho_x|y|)^n\right),
\]
where $\omega':=\min\{\Re(b_1), \Re(c_1-b_1)\}$. Then \eqref{F_{0:1;2}^{1:1;2} series} converges absolutely in the region
\[\DD_F:=\left\{(x,y)\in\CC^2:x\ne 1,\,|\arg(1-x)|<\pi,\,|y|<\min\{1,|1-x|\}\right\},\]
and thus gives an analytic continuation of $F_{0:1;2}^{1:1;2}$ to $\DD_F$.

Replacing ${}_2F_1$ with its Euler integral representation in \eqref{F_{0:1;2}^{1:1;2} series} and then interchanging the order of summation and integration we obtain (initially with $\Re(c_1)>\Re(b_1)>0$ and
$|x|+|y|<1$)
\begin{align}\label{F_{0:1;2}^{1:1;2} integral}
		F_{0:1;2}^{1:1;2}\left[\begin{matrix}
			a_1:b_1;d_1,d_2\\
			-:  c_1;e_1,e_2
		\end{matrix};x,y\right]&=\frac{\Gamma(c_1)}{\Gamma(b_1)\Gamma(c_1-b_1)}\int_0^1 u^{b_1-1}\left(1-u\right)^{c_1-b_1-1}\notag\\
		& \hspace{1cm}\times\left(1-ux\right)^{-a_1}{}_3F_2\left[\begin{matrix}
			d_1,d_2,a_1\\
			e_1,e_2
		\end{matrix};\frac{y}{1-ux}\right]\md u,
\end{align}
where  
	\[
	\Re(c_1)>\Re(b_1)>0,~
	|\arg(1-x)|<\pi~\text{and}~
	\left|\arg\left(1-\frac{y}{1-ux}\right)\right|<\pi ~(\forall u\in[0,1]).
	\]
A simpler condition under which \eqref{F_{0:1;2}^{1:1;2} integral} holds is given by
\begin{equation}\label{F_{0:1;2}^{1:1;2} integral-condition}
\Re(c_1)>\Re(b_1)>0,~
x<1,~|\arg(1-y)|<\pi~\text{and}~
\left|\arg\left(1-\frac{y}{1-x}\right)\right|<\pi.
\end{equation}
The integral \eqref{F_{0:1;2}^{1:1;2} integral} provides another analytic continuation of $F_{0:1;2}^{1:1;2}$.

\subsection{Other functions}
The asymptotic expansions of ${}_1F_1$ are given by (\cite[p. 193]{Andrews-Askey-Roy})
\begin{align}
	{}_1F_1\left[\begin{matrix}
		a\\
		c
	\end{matrix};z\right]&\sim \frac{\Gamma(c)}{\Gamma(c-a)}\sum_{n=0}^{\infty}\frac{\left(a\right)_n\left(1+a-c\right)_n}{n!}\left(-z\right)^{-a-n},\ \ \text{as }\,z\to \infty \,\text{ with }\,\Re(z)<0;\label{1F1 left-plane}\\
	{}_1F_1\left[\begin{matrix}
		a\\
		c
	\end{matrix};z\right]&\sim \frac{\Gamma(c)}{\Gamma(a)}\,\mathrm{e}^z\cdot\sum_{n=0}^{\infty}\frac{\left(1-a\right)_n\left(c-a\right)_n}{n!}\,z^{a-c-n},\ \ \text{as }\,z\to \infty \,\text{ with }\,\Re(z)>0.\label{1F1 right-plane}
\end{align}

As $z\rightarrow\infty$ in the sector $|\arg(z)|\leq \pi-\delta(<\pi)$, 
\begin{equation}\label{Stirling I}
	\Gamma(az+b)\sim \sqrt{2\pi}\mathrm{e}^{-az}(az)^{az+b-(1/2)},
\end{equation}
where $a(>0)$ and $b(\in\mathbb{C})$ are both fixed (see \cite[p. 141, Eq. (5.11.7)]{NIST Handbook}). When $y\rightarrow \pm\infty$, 
\begin{equation}\label{Stirling II}
	|\Gamma(x+\mi y)|\sim\sqrt{2\pi}|y|^{x-(1/2)}\mathrm{e}^{-\pi|y|/2},
\end{equation}
uniformly for bounded real values of $x$ (see \cite[p. 141, Eq. (5.11.9)]{NIST Handbook}).

The Laurent expansion of $\Gamma(z)$ at its pole $z=-n\,(n\in\mathbb{Z}_{\geq 0})$ is given by (\cite[p. 29]{Higher Transcendental Function Vol 1})
\[
\Gamma(z-n)=\frac{\left(-1\right)^n}{n!}\left[z^{-1}+\psi(n+1)\right]+\mo(z),\ \ \text{as}\ \,z\to 0,
\]
where $\psi(x)$ is the digamma function. Hence for $m,n\in\mathbb{Z}_{\geq 0}$, we have
\begin{equation}\label{Gamma limit}
	\lim_{z\to 0}\big[a\,\Gamma(z-m)+b\,\Gamma(-z-n)\big]=\frac{\left(-1\right)^m}{m!}\psi(m+1)\,a+\frac{\left(-1\right)^n}{n!}\psi(n+1)\,b,
\end{equation}
with $a$ and $b$ chosen such that this limit exsits.

Luo and Raina \cite{Luo-Raina 2021}  extended Saran's function $F_K$ to the region
\[
	\mathbb{V}_K=\left\{(x,y,z)\in\CC^3:\begin{array}{c}
		x\ne 1,\ y\ne 1,\ |\arg(1-x)|<\pi,\ |\arg(1-y)|<\pi,\\
		z\ne 0,\ |\arg(-z)|<\pi,\\
		\max\left\{0,\mp\arg(1-x)\right\}+\max\left\{0,\mp\arg(1-y)\right\}<\pi\pm\arg(-z)
	\end{array}\right\}.
\]

\section{Asymptotics of $\Psi_1$}\label{Section 3}
\subsection{Asymptotics for large $y$}
Following \cite{Luo-Raina 2021}, we shall use the Mellin-Barnes contour integral representation of $\Psi_1$ to obtain the asymptotic expansion of $\Psi_1$ for large $y$. The starting point is \eqref{Psi_1 series representation}. Define
\[
\lambda(n):=\frac{\Gamma(a+n)\Gamma(c')}{\Gamma(a)\Gamma(c'+n)}\,_2F_1\left[\begin{matrix}
	a+n,b\\
	c
\end{matrix};x\right]
\]
so that
\[
\Psi_1[a,b;c,c';x,-y]=\sum_{n=0}^{\infty}\frac{\lambda(n)}{n!}(-y)^n.
\]

Using Ramanujan's master theorem \cite[p. 107]{Amdeberhan-Espinosa-Gonzalez}, we formally have the following Mellin transform of the function $\Psi_1$ with respect to $y$:
\begin{equation}\label{Psi_1 Mellin transform}
	\int_0^{\infty}y^{s-1}\Psi_1[a,b;c,c';x,-y]\md y=\Gamma(s)\lambda(-s).
\end{equation}
Applying the inverse Mellin transform to \eqref{Psi_1 Mellin transform}, we arrive at
\[\Psi_1[a,b;c,c';x,-y]=\frac{1}{2\pi\mi}\frac{\Gamma(c')}{\Gamma(a)}\int_{\sigma-\mi\infty}^{\sigma+\mi\infty}{}_2F_1\left[\begin{matrix}
	a-s,b\\
	c
\end{matrix};x\right]\frac{\Gamma(a-s)}{\Gamma(c'-s)}\Gamma(s)y^{-s}\md s,\]
where $\sigma:=\Re(s)$ satisfies: (i) if $\Re(a)>0$ then $0<\sigma<\Re(a)$; (ii) if $\Re(a)\leq0$ then $\sigma=0$.

\begin{theorem}
	Let
	\[\mathbb{V}_{\Psi_1}=\left\{(x,y)\in\CC^2:\begin{array}{c}
		x\ne 1,\ y\ne 0,\ |\arg(1-x)|<\pi,\ |\arg(-y)|<\pi,\\
		\max\left\{0,\mp\arg(1-x)\right\}<\frac{\pi}{2}\pm\arg(-y)
	\end{array}\right\}.\]
	Then
	\begin{equation}\label{Psi_1 Mellin-Barnes integral}
		\Psi_1[a,b;c,c';x,y]=\frac{1}{2\pi\mi}\frac{\Gamma(c')}{\Gamma(a)}\int_{\mathfrak{L}_{\mi\sigma\infty}}{}_2F_1\left[\begin{matrix}
			a+s,b\\
			c
		\end{matrix};x\right]\frac{\Gamma(a+s)}{\Gamma(c'+s)}\Gamma(-s)(-y)^s\md s,
	\end{equation}
	where the path of integration $\mathfrak{L}_{\mi\sigma\infty}$, starting at $\sigma-\mi\infty$ and ending at $\sigma+\mi\infty$, is a straight line parallel to the imaginary axis intended if necessary to separate the poles of $\Gamma(a+s)$ from the poles of $\Gamma(-s)$.
\end{theorem}
\begin{proof}[\textbf{\emph{Proof.}}]
	The proof will be divided into two parts. Firstly, we prove the integral exists for $(x,y)\in\mathbb{V}_{\Psi_1}$. Secondly, we show by using Cauchy's residue theorem that the integral is equal to $\Psi_1$ when $x$ and $y$ are appropriately restricted.
	
	Denote by $\Psi(s)$ the integrand in \eqref{Psi_1 Mellin-Barnes integral}, namely,
	\begin{equation}\label{Psi def}
		\Psi(s):={}_2F_1\left[\begin{matrix}
			a+s,b\\
			c
		\end{matrix};x\right]\frac{\Gamma(a+s)}{\Gamma(c'+s)}\Gamma(-s)(-y)^s.
	\end{equation}
	It follows from \eqref{2F1 imaginary parameter} and \eqref{Stirling II} that as $t\to \pm\infty$,
	\begin{equation}\label{Gamma and 2F1 estimate}
		\begin{split}
			\left|\frac{\Gamma(a+\sigma+\mi t)}{\Gamma(c'+\sigma+\mi t)}\Gamma(-\sigma-\mi t)(-y)^{\sigma+\mi t}\right|& =\mo\left(\left|t\right|^{\Re(a-c')-\sigma-\frac{1}{2}}\mathrm{e}^{-\left(\frac{\pi}{2}\pm\arg(-y)\right)|t|}\right),\\
			\left|{}_2F_1\left[\begin{matrix}
				a+\sigma+\mi t,b\\
				c
			\end{matrix};x\right]\right|& =\mo\left(\left|t\right|^{-\omega}\mathrm{e}^{|t|\max\left\{0,\mp\arg(1-x)\right\}}\right),
		\end{split}
	\end{equation}
	where $\omega=\min\{\Re(b),\Re(c-b)\}$. Now we can conclude that when
	\begin{equation}\label{Psi_1 condition-inequality}
		\max\left\{0,\mp\arg(1-x)\right\}<\frac{\pi}{2}\pm\arg(-y),
	\end{equation}
	the integrand $|\Psi(s)|\ (s=\sigma+\mi t)$ decays exponentially as $|s|\to\infty$ along the path $\mathfrak{L}_{\mi\sigma\infty}$.
	
	The inequality \eqref{Psi_1 condition-inequality} implies that $|\arg(1-x)+\arg(-y)|<\pi/2$. Thus, if we define
	\[\widetilde{\mathbb{V}}_{\Psi_1}:=\left\{(x,y)\in\CC^2:|x|<1,\,y\ne 0,\,|\arg(1-x)|+|\arg(-y)|<\frac{\pi}{2},~|y|<\frac{1}{\mathrm{e}}(1-|x|)\right\},\]
	then we have $\widetilde{\mathbb{V}}_{\Psi_1}\subset\mathbb{V}_{\Psi_1}$.
	
	Next we show that when $(x,y)\in\widetilde{\mathbb{V}}_{\Psi_1}$, the integral in \eqref{Psi_1 Mellin-Barnes integral} yields the series representation \eqref{Psi_1 series representation} of $\Psi_1$. We close the contour by the semi-circle $\mathfrak{C}$ in the right-hand side of the complex plane. The semi-circle $\mathfrak{C}$ is parametrized by 
	$s=(N+1/2)\mathrm{e}^{\mi\theta}$ $(|\theta|\leq \pi/2)$ and $N\to\infty$ through integral values.
	
	From \eqref{Stirling I}, we have immediately
	\[
	\left|\frac{\Gamma(a+s)}{\Gamma(c'+s)}\right| =\mo\left(N^{\Re(a-c')}\right),
	\]
	as $N\rightarrow\infty$. Since $|\sin\pi s|\geq |\sinh\Im(\pi s)|=|\sinh(\pi(N+1/2)\sin\theta)|$, we have
	\[
	|\Gamma(-s)|=\left|\frac{\pi}{\sin\pi s}\frac{1}{\Gamma(1+s)}\right|
	=\mathcal{O}\left(N^{-\frac{1}{2}-\left(N+\frac{1}{2}\right)\cos\theta}\mathrm{e}^{\left(N+\frac{1}{2}\right)(\cos\theta+\theta\sin\theta-\pi|\sin\theta|)}\right),
	\]
	as $N\rightarrow\infty$.
	Thus
	\[
	\left|\frac{\Gamma(a+s)}{\Gamma(c'+s)}\Gamma(-s)(-y)^s\right|
	=\mathcal{O}\left(N^{\Re(a-c')-\frac{1}{2}-\left(N+\frac{1}{2}\right)\cos\theta}\mathrm{e}^{\left(N+\frac{1}{2}\right)(\cos\theta+\theta\sin\theta+\cos\theta\ln|y|-\sin\theta\arg(-y)-\pi|\sin\theta|)}\right),
	\]
	as $N\rightarrow\infty$. Combining this with \eqref{2F1 semi-circle}, we obtain 
	\begin{align*}
	\Psi(s)
	&=\mathcal{O}\left(N^{\Re(a-c')-\frac{1}{2}-\left(N+\frac{1}{2}\right)\cos\theta}\mathrm{e}^{\left(N+\frac{1}{2}\right)(\cos\theta+|\theta||\sin\theta|+\cos\theta\ln|y|+|\sin\theta||\arg(-y)|-\pi|\sin\theta|)}\right)\notag\\
	&\hspace{1cm}\cdot\mo\left(N^{-\omega}\mathrm{e}^{(N+1/2)(|\sin\theta||\arg(1-x)| -\cos\theta\ln(1-|x|)}\right)\notag\\
	&=\mo\left(N^{\Re(a-c')-\omega-\frac{1}{2}-\left(N+\frac{1}{2}\right)\cos\theta}\mathrm{e}^{\left(N+\frac{1}{2}\right)\left(\delta_1|\sin\theta|+\delta_2\cos\theta\right)}\right),
	\end{align*}
	where $\delta_1:=|\theta|+|\arg(-y)|+|\arg(1-x)|-\pi<0$ and $\delta_2:=1+\ln|y|-\ln(1-|x|)<0$. 
	It follows that the integral on the semi-circle $\mathfrak{C}$ tends to zero as $N\to\infty$.
	
	Finally, by Cauchy's residue theorem and the assertion that (\cite[p. 7]{Andrews-Askey-Roy})
	\[\underset{s=-n}{\mathrm{Res}}\,\Gamma(-s)=\frac{\left(-1\right)^n}{n!},\]
	we get the representation \eqref{Psi_1 series representation}. This completes the proof.
\end{proof}

\begin{corollary}
	When $(x,y)\in \mathbb{V}_{\Psi_1}$ and $|y|\to\infty$, then
	\begin{equation}\label{Psi_1 for y in left-plane}
		\Psi_1[a,b;c,c';x,y]\sim \frac{\Gamma(c')}{\Gamma(c'-a)}\sum_{n=0}^{\infty}{}_2F_1\left[\begin{matrix}
			-n,b\\
			c
		\end{matrix};x\right]\frac{(a)_n(1+a-c')_n}{n!}(-y)^{-n-a}.
	\end{equation}
\end{corollary}
\begin{proof}[\textbf{\emph{Proof.}}]
	Let $I$ denote the right-hand side of \eqref{Psi_1 Mellin-Barnes integral}, and let $\Psi(s)$ denote the integrand as defined in \eqref{Psi def}. Choose the positive integer $M\geq \max\{1,\Re(-a)\}$ and shift the integration contour to the left (which is permissible on account of the
	exponential decay of the integrand).
	
	Note that $s=-a-n\ (n\in\ZZ_{\geq 0})$ are the only poles of $\Psi(s)$ because the principal branch of ${}_2F_1[a,b;c;z]$ is an entire function in $a$ (see \cite[p. 384]{NIST Handbook}). Therefore,
	\[
	\underset{s=-a-n}{\mathrm{Res}}\Psi(s)
	={}_2F_1\left[\begin{matrix}
		-n,b\\
		c
	\end{matrix};x\right]\frac{\Gamma(a+n)}{\Gamma(c'-a-n)}\frac{(-1)^n}{n!}(-y)^{-a-n}.
	\]
	By Cauchy's residue theorem, we obtain
	\begin{equation}\label{Psi_1 partial expansion}
		\frac{\Gamma(a)}{\Gamma(c')}I=\sum_{n=0}^{M}{}_2F_1\left[\begin{matrix}
			-n,b\\
			c
		\end{matrix};x\right]\frac{\Gamma(a+n)}{\Gamma(c'-a-n)}\frac{(-1)^n}{n!}(-y)^{-n-a}+R_M(y),
	\end{equation}
	where
	\begin{equation}\label{R_M def}
		R_M(y):=\frac{1}{2\pi\mi}\int_{C_M}\!\Psi(s)\md s
	\end{equation}
	and $C_M$ denotes the vertical line $\Re(s)=\Re(-a)-M-1/2$. Using \eqref{Gamma and 2F1 estimate} can get
	\[
	|\Psi(s)|=\mo\left(|y|^{\Re(-a)-M-\frac{1}{2}}|t|^{\Re(2a-c')+M-\omega}\mathrm{e}^{-\chi|t|}\right), ~
	t=\Im(s),
	\]
	where $\omega$ is given in \eqref{omega-rho def} and 
	\[
	\chi:=\frac{\pi}{2}\pm\arg(-y)-\max\{0,\mp\arg(1-x)\}>0.
	\] 
	Thus the integral \eqref{R_M def} converges when $(x,y)\in \mathbb{V}_{\Psi_1}$. It follows that
	\[R_M(y)=\mo_M\left(|y|^{\Re(-a)-M-\frac{1}{2}}\right),\]
	which shows that \eqref{Psi_1 partial expansion} is an asymptotic expansion of Poincar\'e type. Simplifying \eqref{Psi_1 partial expansion} with the help of the identities $\Gamma(a+n)=\Gamma(a)\left(a\right)_n$ and
	\[
	\Gamma(c'-a-n)=\left(-1\right)^n\frac{\Gamma(c'-a)}{\left(1+a-c'\right)_n} 
	\]
	gives the desired expansion \eqref{Psi_1 for y in left-plane}.
\end{proof}

\begin{remark}
	The region $\mathbb{V}_{\Psi_1}$ has an equivalent form:
	\[
	\mathbb{V}_{\Psi_1}=\left\{(x,y)\in\CC^2:x\ne 1,\,y\ne 0,\,|\arg(-y)|<\frac{\pi}{2},\,|\arg(1-x)+\arg(-y)|<\frac{\pi}{2}\right\},
	\]
	which suggests that \eqref{Psi_1 for y in left-plane} is an asymptotic expansion for $y$ on the left half-plane. If we further let $x=0$ in \eqref{Psi_1 for y in left-plane}, we get the expansion of ${}_1F_1[a;c';y]$ for $\Re(y)<0$, viz. the formula \eqref{1F1 left-plane}.
\end{remark}

In the rest of this subsection, we establish the asymptotics of $\Psi_1$ for $y$ on the right half-plane by applying Olver's result \cite[Theorem 1]{Olver 1970}, which adapts Laplace's method to the complex parameter.

We restrict ourselves to the following three cases:
\begin{itemize}
	\item[(i)] $\Re(c)>\Re(b)>0$, $x=0$, $y\in\mathbb{S}$;
	\item[(ii)] $\Re(c)>\Re(b)>0$, $x<0$, $y\in\mathbb{S}$;
	\item[(iii)] $\Re(c)>\Re(b)>0$, $0<x<1$, $y\in\mathbb{S}$,
\end{itemize}
where $\mathbb{S}:=\{y\in\mathbb{C}:|\arg y|\leq\frac{\pi}{2}-\delta\}$ with $\delta\in(0,\frac{\pi}{2})$ fixed.

\subsubsection{Case\ (i)}
In this case, the asymptotics follows from 
\[
\Psi_1[a,b;c,c';0,y]={}_1F_1\left[\begin{matrix}
	a\\
	c'
\end{matrix};y\right]
\]
and \eqref{1F1 right-plane}.

\subsubsection{Case\ (ii)}

In this case, if $u\in (0,1)$ then $1-ux\in (1,1-x)$. Denote
\[\Psi_1\equiv\Psi_1[a,b;c,c';x,y],\ g(u):=u^{b-1}\left(1-u\right)^{c-b-1}\left(1-ux\right)^{-a}.\]
The expansion \eqref{1F1 right-plane} implies that there exists a constant $C>0$ such that
\begin{equation}\label{1F1 inequality}
	\left|{}_1F_1\left[\begin{matrix}
		a\\
		c'
	\end{matrix};\frac{y}{1-ux}\right]-h(u)\right|\leq CR(u),\ \ y\in\mathbb{S},
\end{equation}
where
\[h(u)=\frac{\Gamma(c')}{\Gamma(a)}\mathrm{e}^{\frac{y}{1-ux}}\left(\frac{y}{1-ux}\right)^{a-c'},\ R(u)=\mathrm{e}^{\frac{\Re(y)}{1-ux}}\left(\frac{|y|}{1-ux}\right)^{\Re(a-c')-1}.\]

For $\Re(\alpha)>0$ and $\Re(\beta)>0$, we define
\[
I_y(\alpha,\beta,\gamma):=\int_0^1 q(u)\mathrm{e}^{-yp(u)}\md u,
\]
where $q(u):=u^{\alpha-1}\left(1-u\right)^{\beta-1}\left(1-ux\right)^\gamma$ and $p(u):=1/(ux-1)$. Obviously, $p(u)$ and $q(u)$ satisfy the conditions \cite[p. 229]{Olver 1970}. Then applying Olver's result to $I_y(\alpha,\beta,\gamma)$, we obtain
\begin{equation}\label{I_y expansion}
	I_y(\alpha,\beta,\gamma)\sim \Gamma(\alpha)y^{-\alpha}\mathrm{e}^y\cdot\sum_{n=0}^{\infty}\frac{\left(\alpha\right)_n\lambda_n}{y^n},\ \ y\to\infty,\,y\in\mathbb{S},
\end{equation}
where $\lambda_0=\left(-x\right)^{-\alpha}$ and
\begin{equation}\label{lambda_n def}
	\lambda_n:=\underset{u=0}{\mathrm{Res}}\,\frac{q(u)}{\left(p(u)+1\right)^{\alpha+n}},\ \ n=0,1,2,\cdots.
\end{equation}

Take the leading term of the expansion \eqref{I_y expansion} and then
\[I_y(\alpha,\beta,\gamma)=\Gamma(\alpha)\left(-x\right)^{-\alpha}y^{-\alpha}\mathrm{e}^y\left[1+\mo\left(y^{-1}\right)\right],\ \ y\to\infty,\,y\in\mathbb{S}.\]
Using \eqref{Psi_1 integral representation} and \eqref{1F1 inequality}, we get that as $y\to \infty$ in $\mathbb{S}$,
\begin{align*}
	\Psi_1& =\frac{\Gamma(c)}{\Gamma(b)\Gamma(c-b)}\int_0^1 g(u)h(u)\md u+\mo\left(\int_0^1|g(u)|R(u)\md u\right)\\
	& =\frac{\Gamma(c)\Gamma(c')}{\Gamma(a)\Gamma(b)\Gamma(c-b)}y^{a-c'}I_y(b,c-b,c'-2a)+\mo\left(\left|y\right|^{\Re(a-c')-1}I_{\Re(y)}\left(\Re(b),\Re(c-b),\Re(c'-2a)+1\right)\right)\\
	& =\frac{\Gamma(c)\Gamma(c')}{\Gamma(a)\Gamma(c-b)}\left(-x\right)^{-b}y^{a-b-c'}\mathrm{e}^y+\mo_x\left(\left|y\right|^{\Re(a-b-c')-1}\mathrm{e}^{\Re(y)}\right).
\end{align*}
Hence, we have the following theorem.

\begin{theorem}
	When $x<0$ and $y\to\infty$ with $y\in\mathbb{S}$, then
	\[
		\Psi_1[a,b;c,c';x,y]\sim \frac{\Gamma(c)\Gamma(c')}{\Gamma(a)\Gamma(c-b)}\left(-x\right)^{-b}y^{a-b-c'}\mathrm{e}^y.
	\]
\end{theorem}

\subsubsection{Case (iii)}
In this case, the pair
\[
(x',y')=\left(\frac{x}{x-1},\frac{y}{1-x}\right)
\] 
satisfies \textbf{Case (ii)}. We can apply the Kummer transformation \eqref{Psi_1 Kummer transformation} to the pair $(x',y')$. Hence we get the following corollary.

\begin{corollary}
	When $0<x<1$ and $y\to\infty$ with $y\in\mathbb{S}$, then
	\begin{equation}\label{Psi_1 Case(iii)}
		\Psi_1[a,b;c,c';x,y]\sim \frac{\Gamma(c)\Gamma(c')}{\Gamma(a)\Gamma(b)}\,x^{b-c}\left(1-x\right)^{c'+2(c-a-b)}y^{a+b-c-c'}\mathrm{e}^{\frac{y}{1-x}}.
	\end{equation}
\end{corollary}

We can obtain the complete asymptotic expansion of $\Psi_1$ in \textbf{Case (ii)}. Indeed, in \eqref{I_y expansion}, take $h(u)$ as the first $N+1$ terms of \eqref{1F1 right-plane} and take $R(u)$ as the corresponding reminder term. Using \eqref{Psi_1 integral representation} and \eqref{I_y expansion}, we get that as $y\to \infty$ in $\mathbb{S}$,
\begin{align*}
	\Psi_1& =\frac{\Gamma(c)}{\Gamma(b)\Gamma(c-b)}\int_0^1 g(u)h(u)\md u+\mo\left(\int_0^1|g(u)|R(u)\md u\right)\\
	& =\frac{\Gamma(c)\Gamma(c')}{\Gamma(a)\Gamma(b)\Gamma(c-b)}\,y^{a-c'}\sum_{n=0}^N\frac{(1-a)_n(c'-a)_n}{n!y^n}I_y(b,c-b,n+c'-2a)+\mo_x\left(|y|^{\Re(a-b-c')-N-1}\mathrm{e}^{\Re(y)}\right)\\
	& =\frac{\Gamma(c)\Gamma(c')}{\Gamma(a)\Gamma(c-b)}\,y^{a-b-c'}\mathrm{e}^y\sum_{n=0}^N\frac{(1-a)_n(c'-a)_n}{n!\,y^n}\sum_{n=0}^N\frac{\left(b\right)_n\lambda_n}{y^n}+\mo_x\left(|y|^{\Re(a-b-c')-N-1}\mathrm{e}^{\Re(y)}\right)\\
	& =\frac{\Gamma(c)\Gamma(c')}{\Gamma(a)\Gamma(c-b)}\,y^{a-b-c'}\mathrm{e}^y\left\{\sum_{n=0}^N\psi_n y^{-n}+\mo_x\left(|y|^{-N-1}\right)\right\},
\end{align*}
where 
\begin{equation}\label{psi_n def}
	\psi_n:=\sum_{k=0}^n\frac{(1-a)_k(c'-a)_k}{k!}(b)_{n-k}\lambda_{n-k},\ n=0,1,2,\cdots,
\end{equation}
and $\lambda_n=\lambda_n(x)$ is given by \eqref{lambda_n def}. In particular, $\psi_0:=\left(-x\right)^{-b}$. Hence, we get the following theorem.

\begin{theorem}
	When $x<0$ and $y\to\infty$ with $y\in\mathbb{S}$, the complete expansion of $\Psi_1$ is given by
	\begin{equation}\label{Psi_1 Case(ii) complete expansion}
		\Psi_1[a,b;c,c';x,y]\sim \frac{\Gamma(c)\Gamma(c')}{\Gamma(a)\Gamma(c-b)}\,y^{a-b-c'}\mathrm{e}^y\sum_{n=0}^{\infty}\psi_n y^{-n},
	\end{equation}
	where $\psi_n=\psi_n(x)$ is given by \eqref{psi_n def}. In particular, $\psi_0=\left(-x\right)^{-b}$.
\end{theorem}

\begin{remark}
	The coefficients in Laplace's method, such as $\lambda_n$ in \eqref{lambda_n def}, were explicitly determined by G. Nemes; see \cite[p. 477, Eq. (2.6)]{Nemes 2013}. To get the complete asymptotic expansion of $\Psi_1$ in \textbf{\emph{Case (iii)}}, we may insert \eqref{Psi_1 Case(ii) complete expansion} into \eqref{Psi_1 Kummer transformation}, but the resulting expansion is terribly complicated.
\end{remark}

\subsection{Asymptotics for large $x$}
In this subsection, we consider the behaviour of $\Psi_1$ for large $x$ and derive the complete asymptotic expansion without using the contour integrals and the residue theorem.

The connection formula for the Gauss hypergeometric function is given by (\cite[p. 390, Eq. (15.8.2)]{NIST Handbook})
\begin{align}\label{2F1 connection formula}
		{}_2F_1\left[\begin{matrix}
			a,b\\
			c
		\end{matrix};z\right]
		&=\mf_c(b,a)\left(-z\right)^{-a}{}_2F_1\left[\begin{matrix}
			1-c+a,a\\
			1-b+a
		\end{matrix};\frac{1}{z}\right]\notag\\
		&\hspace{1cm} +\mf_c(a,b)\left(-z\right)^{-b}{}_2F_1\left[\begin{matrix}
			1-c+b,b\\
			1-a+b
		\end{matrix};\frac{1}{z}\right],\ \ |\arg(-z)|<\pi.
\end{align}
where
\begin{equation}\label{f_gamma def}
	\mf_{\gamma}(a,b):=\frac{\Gamma(\gamma)\Gamma(a-b)}{\Gamma(a)\Gamma(\gamma-b)}.
\end{equation}

Applying \eqref{2F1 connection formula} to \eqref{Psi_1 series representation}, we get 
\begin{equation}\label{Psi_1 connection formula}
	\Psi_1[a,b;c,c';x,y]=\mf_c(b,a)\left(-x\right)^{-a}\bU_1(x,y)+\mf_c(a,b)\left(-x\right)^{-b}\bU_2(x,y),
\end{equation}
where 
\begin{align*}
	\bU_1(x,y)& :=\sum_{n=0}^{\infty}\frac{\left(a\right)_n}{n!\left(c'\right)_n}\frac{\left(1-c+a\right)_n}{\left(1-b+a\right)_n}\,_2F_1\left[\begin{matrix}
		1-c+a+n,a+n\\
		1-b+a+n
	\end{matrix};\frac{1}{x}\right]\left(-\frac{y}{x}\right)^n,\\
	\bU_2(x,y)& :=\sum_{n=0}^{\infty}\frac{\left(a-b\right)_n}{n!\left(c'\right)_n}\,_2F_1\left[\begin{matrix}
		1-c+b,b\\
		1-a-n+b
	\end{matrix};\frac{1}{x}\right]y^n
\end{align*}
and $\mf_\gamma(a,b)$ is given by \eqref{f_gamma def}.

The series $\bU_1$ is a special case of the Kamp\'e de F\'eriet function. In fact, we have
\begin{align*}
	\bU_1(x,y)
	&=\sum_{n=0}^{\infty}\sum_{m=0}^{\infty}\frac{\left(a\right)_{n+m}\left(1-c+a\right)_{n+m}}{\left(1-b+a\right)_{n+m}\left(c'\right)_n m!n!}\frac{(-y)^n}{x^{n+m}} \notag\\
	&=F_{1:0;1}^{2:0;0}
	\left[\begin{matrix}
		a, 1-c+1:-;-\\
		~~~~1-b+a:-;c'
	\end{matrix};\frac{1}{x},-\frac{y}{x}\right],
\end{align*}
which suggests that the double series converges absolutely for $|x|>1$ and $|y|<\infty$. Further calculation gives
\begin{align}
	\bU_1(x,y)
	&=\sum_{k=0}^{\infty}\frac{\left(a\right)_k\left(1-c+a\right)_k}{\left(1-b+a\right)_k k!}\frac{1}{x^k}\sum_{n=0}^k\frac{\left(-k\right)_n}{\left(c'\right)_n}\frac{y^n}{n!}\notag\\
	& =\sum_{k=0}^{\infty}{}_1F_1\left[\begin{matrix}
		-k\\
		c'
	\end{matrix};y\right]\frac{\left(a\right)_k\left(1-c+a\right)_k}{\left(1-b+a\right)_k k!}\frac{1}{x^k},~|x|>1,|y|<\infty.\label{U_1 def}
\end{align}

For the series $\bU_2$, we first note that
\begin{align*}
	\bU_2(x,y)& =\sum_{\ell,n=0}^{\infty}\frac{\left(1-c+b\right)_{\ell}\left(b\right)_{\ell}}{\left(1-a+b-n\right)_{\ell}\ell!}\frac{\left(a-b\right)_n}{\left(c'\right)_n}\frac{1}{x^{\ell}}\frac{y^n}{n!}\notag\\
	&=\sum_{n,\ell=0}^{\infty}(a-b)_{n-\ell}\frac{\left(1-c+b\right)_{\ell}\left(b\right)_{\ell}}{(c')_n}\frac{(-1/x)^{\ell}}{\ell!}\frac{y^n}{n!}\notag\\
	&=H_{11}\left[a-b,b,1-c+b,c';y,-\frac{1}{x}\right], ~|x|>1,|y|<\infty, 
\end{align*}
where $H_{11}$ defined by 
\[
H_{11}\left[a,b,c,d;x,y\right]:=\sum_{m,n=0}^{\infty}\frac{(a)_{m-n}(b)_n(c)_n}{(d)_m}\frac{x^m}{m!}\frac{y^n}{n!}, ~ |x|<\infty, |y|<1
\]
is one member of \emph{Horn's list} (see \cite[Eq. (39), p. 227]{Higher Transcendental Function Vol 1} and \cite{Brychkov-Savischenko-2023}). In addition, we have
\begin{align}
	\bU_2(x,y)
	&=\sum_{\ell,n=0}^{\infty}\frac{\left(1-c+b\right)_{\ell}\left(b\right)_{\ell}}{\left(1-a+b\right)_{\ell}\ell!}\frac{1}{x^{\ell}}\cdot\frac{\left(a-b-\ell\right)_n}{\left(c'\right)_n}\frac{y^n}{n!}\notag\\
	& =\sum_{\ell=0}^{\infty}{}_1F_1\left[\begin{matrix}
		a-b-\ell\\
		c'
	\end{matrix};y\right]\frac{\left(1-c+b\right)_{\ell}\left(b\right)_{\ell}}{\left(1-a+b\right)_{\ell}\ell!}\frac{1}{x^{\ell}}, ~|x|>1,|y|<\infty.\label{U_2 def}
\end{align}

The previous derivations are summarized in the following theorem.
\begin{theorem}
	When $|\arg(-x)|<\pi$ and $|x|\to \infty$, we have
	\[
	\Psi_1[a,b;c,c';x,y]\sim\mf_c(b,a)\left(-x\right)^{-a}+\mf_c(a,b)\left(-x\right)^{-b}{}_1F_1\left[\begin{matrix}
		a-b\\
		c'
	\end{matrix};y\right],
	\]
	where $\mf_{\gamma}$ is given by \eqref{f_gamma def}. The complete expansion is obtained by combining \eqref{Psi_1 connection formula}, \eqref{U_1 def} and \eqref{U_2 def}.
\end{theorem}

\section{Asymptotics of $F_K$}\label{Section 4}

Using \eqref{1F1 left-plane} and \eqref{Psi_1 for y in left-plane}, we get a sufficient condition of convergence for the integral \eqref{F_K Laplace integral}:
\[
	\Re(\alpha_2)>0,~|\arg(-y)|<\frac{\pi}{2}
		~\text{and}~
		(x,z)\in\mathbb{V}_{\Psi_1}.
\]
Under this condition, we also have $(x,y,z)\in\mathbb{V}_K$. So the integral \eqref{F_K Laplace integral} makes sense.

After the change of the variables $s=\widetilde{y}t$, with $\widetilde{y}=1/|y|$, and defining $\gamma:=\widetilde{z}/\widetilde{y}:=|y|/|z|$, the integral \eqref{F_K Laplace integral} reads
\begin{align}\label{F_K another integral}
		&F_K[\alpha_1,\alpha_2,\alpha_2,\beta_1,\beta_2,\beta_1;\gamma_1,\gamma_2,\gamma_3;x,y,z]\notag\\
		&\hspace{1cm} =\frac{\widetilde{y}}{\Gamma(\alpha_2)}\int_0^{\infty}\mathrm{e}^{-\widetilde{y}t}\left(\widetilde{y}t\right)^{\alpha_2-1}
		{}_1F_1\left[\begin{matrix}
			\beta_2\\
			\gamma_2
		\end{matrix};t\mathrm{e}^{\mi\alpha}\right]
		\Psi_1\big[\beta_1,\alpha_1;\gamma_1,\gamma_3;x,t\mathrm{e}^{\mi\beta}/\gamma\big]\md t,
\end{align}
where $\alpha$ and $\beta$ are the principal arguments of the respective variables $y$ and $z$: $y=|y|\mathrm{e}^{\mi \alpha}$ and $z=|z|\mathrm{e}^{\mi\beta}$.

We are interested in the approximation of $F_K$ for large $y$ and $z$ uniformly in $\gamma$ with 
\[
0<\varrho_1<\gamma<\varrho_2<\infty.
\] 
This means that $\widetilde{y}$ is small, $\gamma$ is bounded, and the remaining parameters $\alpha_1,\alpha_2,\beta_1,\beta_2,\gamma_1,\gamma_2,\gamma_3,x$ are fixed. In order to apply L\'{o}pez's theorem \cite{Lopez 2008}, we also require that 
\[
\alpha_2>0~\text{and}~\beta_1+\beta_2>0.
\]

Throughout this section, we define
\begin{equation}\label{Functions f and h}
	f(t):={}_1F_1\left[\begin{matrix}
		\beta_2\\
		\gamma_2
	\end{matrix};t\mathrm{e}^{\mi\alpha}\right]\Psi_1\big[\beta_1,\alpha_1;\gamma_1,\gamma_3;x,t\mathrm{e}^{\mi\beta}/\gamma\big],\ h(t):=\mathrm{e}^{-t}t^{\alpha_2-1}.
\end{equation}
Then the integral \eqref{F_K another integral} is of the form of the Mellin integrals considered in \cite{Lopez 2008}:
\begin{equation}\label{F_K Mellin integral}
	F_K[\alpha_1,\alpha_2,\alpha_2,\beta_1,\beta_2,\beta_1;\gamma_1,\gamma_2,\gamma_3;x,y,z]=\frac{\widetilde{y}}{\Gamma(\alpha_2)}\int_0^{\infty}h(\widetilde{y}t)f(t)\md t.
\end{equation}

The functions $f(t)$ and $h(t)$ are locally integrable on $[0,\infty)$, a property required by the technique introduced in \cite{Lopez 2008} to obtain an asymptotic expansion of the integral \eqref{F_K Mellin integral} for small $\widetilde{y}$. That technique also requires the asymptotic behaviours of these functions given below and is valid for $\Re(y)<0$ and $\Re(z)<0\ (\pi/2<|\alpha|,|\beta|<\pi)$:
\begin{itemize}
	\item[(i)] A power asymptotic expansion of $h(t)$ at $t=0$:
	\[
		h(t)=\sum_{k=0}^{m-1}A_k t^{k+\alpha_2-1}+h_m(t),\ \ A_k:=\frac{\left(-1\right)^k}{k!},
	\]
	with $h_m(t)=\mo\big(t^{m+\alpha_2-1}\big)$ when $t\to 0^+$;
	\item[(ii)] An inverse power asymptotic expansion of $f(t)$ at $t=\infty$:
	\[
		f(t)=\sum_{k=0}^{n-1}B_k t^{-k-\beta_1-\beta_2}+f_n(t),
	\]
	with $B_k=\left(-y\widetilde{y}\right)^{-k-\beta_1-\beta_2}\widetilde{B}_k$,
	\begin{align*}
		\widetilde{B}_k:={}&\frac{\Gamma(\gamma_2)\Gamma(\gamma_3)}{\Gamma(\gamma_2-\beta_2)\Gamma(\gamma_3-\beta_1)}\left(\frac{y}{z}\right)^{\beta_1}\\
		& \times\sum_{j=0}^k\frac{(\beta_2)_j(1+\beta_2-\gamma_2)_j(\beta_1)_{k-j}(1+\beta_1-\gamma_3)_{k-j}}{j!(k-j)!}\,_2F_1\left[\begin{matrix}
			j-k,\alpha_1 \\
			\gamma_1
		\end{matrix};x\right]\left(\frac{y}{z}\right)^{k-j},
	\end{align*}
	and $f_n(t)=\mo\big(t^{-\beta_1-\beta_2-n}\big)$ when $t\to \infty$. The above expansion follows from \eqref{1F1 left-plane} and \eqref{Psi_1 for y in left-plane};
	\item[(iii)] $h(t)=\mo\big(t^{-N}\big)$ when $t\to\infty$ for any positive integer $N$ and $f(t)=1+\mo(t)$ when $t\to 0^+$.
\end{itemize}

Following \cite{Lopez 2008}, we denote by $\mathfrak{M}[g;s]$ the Mellin transform of $g$. Clearly, $\mathfrak{M}[h;s]=\Gamma(s+\alpha_2-1)$. Using \eqref{M[Psi_F;s] evaluation}, we can readily get
\begin{align*}
		\mathfrak{M}[f;s]
		=(-y\widetilde{y})^{-s}&\Bigg\{ C_1^*(s)\left(\frac{y}{z}\right)^{s-\beta_2}F_{0:1;2}^{1:1;2}\left[\begin{matrix}
			\beta_1+\beta_2-s:&\!\!\!\alpha_1;&\!\!\!\beta_2,\,\beta_2-\gamma_2+1\\[-0.6ex]
			\rule[0.8mm]{4.9em}{0.5pt}:      &\!\!\!\gamma_1;&\!\!\!\beta_2+1-s,\,\beta_2+\gamma_3-s
		\end{matrix};\,x,\,-\frac{z}{y}\right]\notag\\
		&+C_2^*(s)F_{0:1;2}^{1:1;2}\left[\begin{matrix}
			\beta_1:&\!\!\!\alpha_1;&\!\!\!s,\,s-\gamma_2+1\\[-0.5ex]
			\rule[0.8mm]{1.1em}{0.5pt}:&\!\!\!\gamma_1;&\!\!\!\gamma_3,\,s-\beta_2+1
		\end{matrix};\,x,\,-\frac{z}{y}\right]\Bigg\},
\end{align*}
where
\begin{equation}\label{C_1^*(s) and C_2^*(s) def}
	C_1^*(s):=\frac{\Gamma(\gamma_2)\Gamma(\gamma_3)\Gamma(s-\beta_2)\Gamma(\beta_1+\beta_2-s)}{\Gamma(\beta_1)\Gamma(\gamma_2-\beta_2)\Gamma(\beta_2+\gamma_3-s)},\ \ C_2^*(s):=\frac{\Gamma(\gamma_2)\Gamma(s)\Gamma(\beta_2-s)}{\Gamma(\beta_2)\Gamma(\gamma_2-s)},
\end{equation}
$\Re(y)<0$, $(x,z)\in \mathbb{V}_{\Psi_1}$ and $\rho_x<|y|/|z|$.

Let us summarize the conditions needed in our further derivation as follows 
\begin{equation}\label{F_K final condition}
	\alpha_2>0,~\beta_1+\beta_2>0,~|\arg(-y)|<\frac{\pi}{2},~ (x,z)\in\mathbb{V}_{\Psi_1}~\text{and}~ \rho_x<\frac{|y|}{|z|}<\varrho_2<\infty.
\end{equation}

\subsection{The nonlogarithmic case}
When $\beta_1+\beta_2-\alpha_2\notin \ZZ$, from \cite[Theorem 1]{Lopez 2008} we have the following asymptotic expansion of $F_K$ for small $\widetilde{y}$. For any $n\in\ZZ_{>0}$,
\begin{equation}\label{nonlogarithmic formula}
	\begin{split}
		\Gamma(\alpha_2)F_K&[\alpha_1,\alpha_2,\alpha_2,\beta_1,\beta_2,\beta_1;\gamma_1,\gamma_2,\gamma_3;x,y,z]\\
		& =\sum_{k=0}^{n-1}B_k \mathfrak{M}[h;1-k-\beta_1-\beta_2]\widetilde{y}^{\,k+\beta_1+\beta_2}+\sum_{k=0}^{m-1}A_k \mathfrak{M}[f;k+\alpha_2]\widetilde{y}^{\,k+\alpha_2}+R_n\left(\widetilde{y}\right),
	\end{split}
\end{equation}
where $m=n+\lfloor 1-\alpha_2+\beta_1+\beta_2\rfloor$ and the reminder $R_n\left(\widetilde{y}\right)=\mo\big(\widetilde{y}^{\,n+\beta_1+\beta_2}\big)$ when $\widetilde{y}\to 0$.

Computing the Mellin transforms involved in \eqref{nonlogarithmic formula}, we obtain that, under the condition \eqref{F_K final condition},
\begin{equation}\label{F_K nonlogarithmic case asymptotics}
	F_K[\alpha_1,\alpha_2,\alpha_2,\beta_1,\beta_2,\beta_1;\gamma_1,\gamma_2,\gamma_3;x,y,z]=\sum_{k=0}^{n-1}\frac{\widehat{B}_k(y/z)}{y^{k+\beta_1+\beta_2}}+\sum_{k=0}^{m-1}\frac{\widehat{A}_k(y/z)}{y^{k+\alpha_2}}+R_n(y),
\end{equation}
with
\begin{align}\label{widehat{A}_k def}
		\widehat{A}_k(y/z)&:=\frac{\left(-1\right)^{-k-\alpha_2}}{\Gamma(\alpha_2)}\Gamma(\gamma_2)A_k\,\Bigg\{ \frac{\Gamma(\gamma_3)\Gamma(k+\alpha_2-\beta_2)\Gamma(\beta_1+\beta_2-\alpha_2-k)}{\Gamma(\beta_1)\Gamma(\gamma_2-\beta_2)\Gamma(\beta_2+\gamma_3-\alpha_2-k)}\left(\frac{y}{z}\right)^{k+\alpha_2-\beta_2}\notag\\
		&\hspace{0.5cm}\times F_{0:1;2}^{1:1;2}\left[\begin{matrix}
			\beta_1+\beta_2-\alpha_2-k:&\!\!\!\alpha_1;&\!\!\!\beta_2,\,\beta_2-\gamma_2+1\\[-0.4ex]
			\rule[0.8mm]{7.2em}{0.5pt}:      &\!\!\!\gamma_1;&\!\!\!\beta_2+1-\alpha_2-k,\,\beta_2+\gamma_3-\alpha_2-k
		\end{matrix};\,x,\,-\frac{z}{y}\right]\notag\\
		&\hspace{0.5cm} +\frac{\Gamma(k+\alpha_2)\Gamma(\beta_2-\alpha_2-k)}{\Gamma(\beta_2)\Gamma(\gamma_2-\alpha_2-k)}F_{0:1;2}^{1:1;2}\left[\begin{matrix}
			\beta_1:&\!\!\!\alpha_1;&\!\!\!k+\alpha_2,\,k+\alpha_2-\gamma_2+1\\[-0.4ex]
			\rule[0.8mm]{1.1em}{0.5pt}:&\!\!\!\gamma_1;&\!\!\!\gamma_3,\,k+\alpha_2-\beta_2+1
		\end{matrix};\,x,\,-\frac{z}{y}\right]\Bigg\}
\end{align}
and
\begin{equation}\label{widehat{B}_k def}
	\widehat{B}_k(y/z):=\frac{\Gamma(\alpha_2-\beta_1-\beta_2-k)}{\Gamma(\alpha_2)}\left(-1\right)^{-k-\beta_1-\beta_2}\widetilde{B}_k.
\end{equation}
The reminder $R_n(y)$ verifies $R_n(y)=\mo\big(y^{-n-\beta_1-\beta_2}\big)$ when $|y|\to\infty$ uniformly in $|z|$ with $y/z$ bounded.

\subsection{The logarithmic case I}
When $1+\beta_1+\beta_2-\alpha_2\in\ZZ_{>0}$, from \cite[Theorem 1]{Lopez 2008} we have the following asymptotic expansion for small $\widetilde{y}$. For any $n\in\ZZ_{>0}$,
\begin{align}\label{logarithmic case I formula}
		&\Gamma(\alpha_2)F_K[\alpha_1,\alpha_2,\alpha_2,\beta_1,\beta_2,\beta_1;\gamma_1,\gamma_2,\gamma_3;x,y,z]=\sum_{k=0}^{\beta_1+\beta_2-\alpha_2-1}A_k \mathfrak{M}[f;k+\alpha_2]\widetilde{y}^{\,k+\alpha_2}\notag\\
		&\hspace{1cm}
		+\sum_{k=0}^{n-1}\widetilde{y}^{\,k+\beta_1+\beta_2}\Bigg\{\!-B_k A_{k-\alpha_2+\beta_1+\beta_2}\log\widetilde{y}+\lim_{s\to 0}\Big(B_k \mathfrak{M}[h;s+1-k-\beta_1-\beta_2]\notag\\
		&\hspace{1cm} +A_{k-\alpha_2+\beta_1+\beta_2}\mathfrak{M}[f;s+k+\beta_1+\beta_2]\Big)\Bigg\}+R_n(\widetilde{y}),
\end{align}
where $m=n-\alpha_2+\beta_1+\beta_2$ and the reminder $R_n(\widetilde{y})=\mo\big(\widetilde{y}^{\,n+\beta_1+\beta_2}\log\widetilde{y}\big)$ when $\widetilde{y}\to 0$.

To compute the limit involved in \eqref{logarithmic case I formula}, we follow  \cite[Eq. (9) and (10)]{Lopez 2008}. Write
\[\mathfrak{M}[f;s+k+\beta_1+\beta_2]=C_f-\frac{B_k}{s},\ \mathfrak{M}[h;s+1-k-\beta_1-\beta_2]=C_h+\frac{A_{k-\alpha_2+\beta_1+\beta_2}}{s},\ \ \text{as}\ \,s\to 0.\]
Thus, that limit exsits. Further, observing that only the factor $\Gamma(\beta_1+\beta_2-s)$ of $C_1^*(s)$ (see \eqref{C_1^*(s) and C_2^*(s) def}) and $\mathfrak{M}[h;s]$ have poles at $1-k-\beta_1-\beta_2$, that limit can be reduced to the form as described in \eqref{Gamma limit}. Therefore, using \eqref{Gamma limit} and after some calculations, we obtain that, under the condition \eqref{F_K final condition},
\begin{align}\label{F_K logarithmic case I asymptotics}
		& F_K[\alpha_1,\alpha_2,\alpha_2,\beta_1,\beta_2,\beta_1;\gamma_1,\gamma_2,\gamma_3;x,y,z]\notag\\
		&\hspace{1cm} =\sum_{k=0}^{\beta_1+\beta_2-\alpha_2-1}\frac{\widehat{A}_k(y/z)}{y^{k+\alpha_2}}+\sum_{k=0}^{n-1}\frac{\widehat{C}_k(y/z)+B_k^{\times}(y/z)\log|y|}{y^{k+\beta_1+\beta_2}}+R_n(y),
\end{align}
with $\widehat{A}_k(y/z)$ given in \eqref{widehat{A}_k def},
\begin{equation}\label{B_k times def}
	B_k^{\times}(y/z):=\frac{\left(-1\right)^{-k-\beta_1-\beta_2}}{\Gamma(\alpha_2)}\widetilde{B}_k A_{k-\alpha_2+\beta_1+\beta_2}
\end{equation}
and
\begin{align}\label{widehat{C}_k def}
		&\widehat{C}_k(y/z):=\frac{\left(-1\right)^{-k-\beta_1-\beta_2}}{\Gamma(\alpha_2)}\Bigg\{\frac{\widetilde{B}_k\left(-1\right)^{k-\alpha_2+\beta_1+\beta_2}\psi(k-\alpha_2+\beta_1+\beta_2+1)}{(k-\alpha_2+\beta_1+\beta_2)!}\notag\\
		&\hspace{0.5cm} +A_{k-\alpha_2+\beta_1+\beta_2}\Gamma(\gamma_2)\Bigg[\frac{\left(y/z\right)^{k+\beta_1}\left(-1\right)^k\psi(k+1)\Gamma(\gamma_3)\Gamma(k+\beta_1)}{\Gamma(\beta_1)\Gamma(\gamma_2-\beta_2)\Gamma(\gamma_3-\beta_1-k)k!}\notag\\
		&\hspace{0.5cm}\times  F_{0:1;2}^{1:1;2}\left[\begin{matrix}
			-k:&\!\!\!\alpha_1;&\!\!\!\beta_2,\,\beta_2-\gamma_2+1\\[-0.5ex]
			\rule[0.8mm]{1.4em}{0.5pt}:      &\!\!\!\gamma_1;&\!\!\!1-\beta_1-k,\,\gamma_3-\beta_1-k
		\end{matrix};\,x,\,-\frac{z}{y}\right]\notag\\
		&\hspace{0.5cm} +\frac{\Gamma(-\beta_1-k)\Gamma(k+\beta_1+\beta_2)}{\Gamma(\beta_2)\Gamma(\gamma_2-\beta_1-\beta_2-k)}F_{0:1;2}^{1:1;2}\left[\begin{matrix}
			\beta_1:&\!\!\!\alpha_1;&\!\!\!k+\beta_1+\beta_2,\,k+\beta_1+\beta_2-\gamma_2+1\\[-0.5ex]
			\rule[0.8mm]{1.1em}{0.5pt}:&\!\!\!\gamma_1;&\!\!\!\gamma_3,\,k+\beta_1+1
		\end{matrix};\,x,\,-\frac{z}{y}\right]\Bigg]\Bigg\},
\end{align}
where $\psi(s)$ is the digamma function.

\subsection{The logarithmic case II}
When $\alpha_2-\beta_1-\beta_2\in\ZZ_{>0}$, from \cite[Theorem 1]{Lopez 2008} we have the following asymptotic expansion for small $\widetilde{y}$. For any $n\in\ZZ_{>0}$,
\begin{align}\label{logarithmic case II formula}
		&\Gamma(\alpha_2)F_K[\alpha_1,\alpha_2,\alpha_2,\beta_1,\beta_2,\beta_1;\gamma_1,\gamma_2,\gamma_3;x,y,z]=\sum_{k=0}^{\alpha_2-\beta_1-\beta_2-1}B_k \mathfrak{M}[h;1-k-\beta_1-\beta_2]\widetilde{y}^{\,k+\beta_1+\beta_2}\notag\\
		&\hspace{3cm} +\sum_{k=0}^{m-1}\widetilde{y}^{\,k+\alpha_2}\Bigg\{\!-A_k B_{k+\alpha_2-\beta_1-\beta_2}\log\widetilde{y}+\lim_{s\to 0}\Big(A_k \mathfrak{M}[f;s+k+\alpha_2]\notag\\
		&\hspace{3cm} +B_{k+\alpha_2-\beta_1-\beta_2}\mathfrak{M}[h;s+1-\alpha_2-k]\Big)\Bigg\}+R_n(\widetilde{y}),
\end{align}
where $n=m+\alpha_2-\beta_1-\beta_2$ and the reminder $R_n(\widetilde{y})=\mo\big(\widetilde{y}^{\,n+\beta_1+\beta_2}\log\widetilde{y}\big)$ when $\widetilde{y}\to 0$. Computing the limit involved in \eqref{logarithmic case II formula} and after some calculations, we obtain that, under the condition \eqref{F_K final condition},
\begin{align}\label{F_K logarithmic case II asymptotics}
		& F_K[\alpha_1,\alpha_2,\alpha_2,\beta_1,\beta_2,\beta_1;\gamma_1,\gamma_2,\gamma_3;x,y,z]=\sum_{k=0}^{\alpha_2-\beta_1-\beta_2-1}\frac{\widehat{B}_k(y/z)}{y^{k+\beta_1+\beta_2}}\notag\\
		&\hspace{1cm} +\sum_{k=0}^{m-1}\frac{\widehat{C}_{k+\alpha_2-\beta_1-\beta_2}(y/z)+B_{k+\alpha_2-\beta_1-\beta_2}^{\times}(y/z)\log|y|}{y^{k+\alpha_2}}+R_n(y),
\end{align}
with $\widehat{B}_k(y/z),\,B_k^{\times}(y/z)$ and $\widehat{C}_k(y/z)$ given in \eqref{widehat{B}_k def}, \eqref{B_k times def} and \eqref{widehat{C}_k def}, respectively. In \eqref{F_K logarithmic case I asymptotics} and \eqref{F_K logarithmic case II asymptotics}, the reminder term $R_n(y)=\mo\big(y^{-m-\alpha_2}\log y\big)$ when $|y|\to\infty$ uniformly in $|z|$ with $y/z$ bounded.

\begin{remark}
	When $x<1$, in view of \eqref{F_{0:1;2}^{1:1;2} integral-condition}, the condition \eqref{F_K final condition} attached to the functions $f$ and $h$ defined in \eqref{Functions f and h} can be replaced by
	\begin{align*}
	&\alpha_2>0, ~\beta_1+\beta_2>0, ~
	|\arg(-y)|<\frac{\pi}{2}, ~|\arg(-z)|<\frac{\pi}{2}, \\
	&\Re(\gamma_1)>\Re(\alpha_1)>0,~x<1,~\left|\arg\left(1-\frac{y}{1-x}\right)\right|<\pi,~
	0<\varrho_1<\frac{|y|}{|z|}<\varrho_2<\infty. 
	\end{align*} 
	Under this condition, the asymptotic expansions \eqref{F_K nonlogarithmic case asymptotics}, \eqref{F_K logarithmic case I asymptotics} and \eqref{F_K logarithmic case II asymptotics} and the corresponding error estimates remain true. In addition, by letting $x=0$ in \eqref{F_K Mellin integral}, we obtain
		\[
			F_2[\alpha_2,\beta_2,\beta_1;\gamma_2,\gamma_3;y,z]=\frac{\widetilde{y}}{\Gamma(\alpha_2)}\int_0^{\infty}h(\widetilde{y}t)f(t)\md t,
		\]
		where
		\[f(t)={}_1F_1\left[\begin{matrix}
			\beta_2\\
			\gamma_2
		\end{matrix};t\mathrm{e}^{\mi\alpha}\right]
		{}_1F_1\left[\begin{matrix}
			\beta_1\\
			\gamma_3
		\end{matrix};t\mathrm{e}^{\mi\beta}/\gamma\right],\ h(t)=\mathrm{e}^{-t}t^{\alpha_2-1}
		\]
		and $F_2$ is the second Appell hypergeometric function. The integral shares the same form with the Mellin integral considered by Garcia and L\'{o}pez in \emph{\cite[Eq. (1.4)]{Garcia-Lopez 2010}}. Therefore, our results contain Garcia and L\'{o}pez's results as special cases. 
\end{remark}

\appendix
\section{Calculation of a Mellin transform}\label{Appendix A}
Here we calculate a Mellin transform needed in Section \ref{Section 4}. As in Section \ref{Section 4}, the symbol $\mathfrak{M}[g;s]$ denotes the Mellin transform of $g$ when it exists, or its analytic continuation as a function of $s$.

When $\Re(\lambda)<0$ and $(x,\mu)\in\mathbb{V}_{\Psi_1}$, we define
\[
\Psi_F(t):={}_1F_1
\left[\begin{matrix}
	\xi\\
	\eta
\end{matrix};\lambda t\right]
\Psi_1[a,b;c,c';x,\mu t].
\]
In view of \eqref{1F1 left-plane} and \eqref{Psi_1 for y in left-plane}, $\Psi_F(t)=1+o(1)$ as $t\to 0^+$ and $\Psi_F$ has an inverse power asymptotic expansion as $t\to +\infty$. Moreover, $\Psi_F\in L_{\text{loc}}^1(0,\infty)$. Thanks to \cite[p. 278, Remark 1]{Lopez 2008}, the Mellin transform $\mathfrak{M}[\Psi_F;s]$ exists and defines a meromorphic function of $s$ in the half plane $\Re(s)>0$.

Using \eqref{Psi_1 integral representation} and swapping the order of integrations, we obtain
\begin{align}\label{Psi_F Mellin transform}
		\mathfrak{M}[\Psi_F;s]
		&=\frac{\Gamma(c)}{\Gamma(b)\Gamma(c-b)}\int_0^1  u^{b-1}\left(1-u\right)^{c-b-1}\left(1-ux\right)^{-a}\notag\\
		&\hspace{1cm}\times\mathfrak{M}\left[{}_1F_1\left[\begin{matrix}
			\xi\\
			\eta
		\end{matrix};\lambda t\right]{}_1F_1\left[\begin{matrix}
		a\\
		c'
		\end{matrix};\frac{\mu t}{1-ux}\right];s\right]\md u.
\end{align}
The Mellin transform in right-hand side of \eqref{Psi_F Mellin transform} is given by \cite[p. 436, Eq. (1)]{Handbook of Mellin transforms}
\begin{align}\label{general Psi_F Mellin transform}
		&\mathfrak{M}\left[
		{}_1F_1\left[\begin{matrix}
			a\\
			b
		\end{matrix};-\omega t\right]
		{}_1F_1\left[\begin{matrix}
			c\\
			d
		\end{matrix};-\sigma t\right];s\right]
		=\omega^{-s}
		\frac{\Gamma(b)\Gamma(s)\Gamma(a-s)}{\Gamma(a)\Gamma(b-s)}
		{}_3F_2\left[\begin{array}{c}
			c,s,s-b+1\\
			d,s-a+1
		\end{array};-\frac{\sigma}{\omega}\right]\notag\\
		&\hspace{2cm} +\sigma^{a-s}\omega^{-a}
		\frac{\Gamma(b)\Gamma(d)\Gamma(s-a)\Gamma(a+c-s)}{\Gamma(d)\Gamma(b-a)\Gamma(a+d-s)}
		{}_3F_2\left[\begin{array}{c}
			a,a-b+1,a+c-s\\
			a-s+1,a+d-s
		\end{array};-\frac{\sigma}{\omega}\right],
\end{align}
where $\Re(\sigma)>0$, $\Re(\omega)>0$ and $0<\Re(s)<\Re(a+c)$. In order to apply \eqref{general Psi_F Mellin transform} to \eqref{Psi_F Mellin transform}, we need the restriction $\Re(\mu/(1-ux))<0$. Here for convenience, we require that $x<1$ so that
\[
\Re\left(\frac{\mu}{1-ux}\right)=\frac{1}{1-ux}\Re(\mu)
=-\frac{1}{1-ux}|\mu|\cos\left(\arg(-\mu)\right)<0.
\]

Taking $\omega=-\lambda$ and $\sigma=\mu/(ux-1)$ in \eqref{general Psi_F Mellin transform} and inserting \eqref{general Psi_F Mellin transform} into \eqref{Psi_F Mellin transform}, we obtain
\[
	\mathfrak{M}[\Psi_F;s]=C_1(s)(-\mu)^{-s}\left(\frac{\mu}{\lambda}\right)^{\xi}I_1(s)+C_2(s)(-\lambda)^{-s}I_2(s),
\]
where
\begin{align*}
	C_1(s)&:=
	\frac{\Gamma(\eta)\Gamma(c')\Gamma(s-\xi)\Gamma(a+\xi-s)}{\Gamma(a)\Gamma(\eta-\xi)\Gamma(\xi+c'-s)},~~~
	C_2(s):=
	\frac{\Gamma(\eta)\Gamma(s)\Gamma(\xi-s)}{\Gamma(\xi)\Gamma(\eta-s)},\\
	I_1(s)
	&:=\frac{\Gamma(c)}{\Gamma(b)\Gamma(c-b)}\int_0^1 u^{b-1}\left(1-u\right)^{c-b-1}\left(1-ux\right)^{s-a-\xi}\notag\\
	&\hspace{2cm}\times {}_3F_2\left[\begin{matrix}
			\xi,\xi-\eta+1,a+\xi-s\\
			\xi+1-s,\xi+c'-s
		\end{matrix};\frac{\mu}{\lambda}\frac{1}{ux-1}\right]\md u
\end{align*}
and
\begin{align*}
		I_2(s)
		:=\frac{\Gamma(c)}{\Gamma(b)\Gamma(c-b)}&\int_0^1 u^{b-1}\left(1-u\right)^{c-b-1}\left(1-ux\right)^{-a}\notag\\
		&\hspace{0.5cm}\times{}_3F_2\left[\begin{matrix}
			a,s,s-\eta+1\\
			c',s-\xi+1
		\end{matrix};\frac{\mu}{\lambda}\frac{1}{ux-1}\right]\md u.
\end{align*}
With the help of \eqref{F_{0:1;2}^{1:1;2} integral}, $I_1(s)$ and $I_2(s)$ can be evaluated in terms of the Kamp\'e de F\'eriet function $F_{0:1;2}^{1:1;2}$. 
\begin{theorem}
	\begin{align}\label{M[Psi_F;s] evaluation}
		\mathfrak{M}[\Psi_F;s]&
		= C_1(s)(-\mu)^{-s}\left(\frac{\mu}{\lambda}\right)^{\xi}F_{0:1;2}^{1:1;2}\left[\begin{matrix}
			a+\xi-s:&\!\!\!b;&\!\!\!\xi,\,\xi-\eta+1\\[-0.4ex]
			\rule[0.8mm]{3.7em}{0.5pt}:      &\!\!\!c;&\!\!\!\xi+1-s,\,\xi+c'-s
		\end{matrix};\,x,\,-\frac{\mu}{\lambda}\right]\notag\\
		&\hspace{1cm}+C_2(s)\left(-\lambda\right)^{-s}F_{0:1;2}^{1:1;2}\left[\begin{matrix}
			a:&\!\!\!b;&\!\!\!s,s-\eta+1\\[-0.4ex]
			\rule[0.8mm]{0.6em}{0.5pt}:&\!\!\!c;&\!\!\!c',s-\xi+1
		\end{matrix};\,x,\,-\frac{\mu}{\lambda}\right].
	\end{align}
\end{theorem}

By continuation of $F_{0:1;2}^{1:1;2}$, \eqref{M[Psi_F;s] evaluation} is valid for
\begin{equation}\label{M[Psi_F;s] condition}
	\Re(\lambda)<0,\ (x,\mu)\in\mathbb{V}_{\Psi_1}~ \text{and}~ \left(x,-\frac{\mu}{\lambda}\right)\in\DD_F.
\end{equation}

\begin{remark} 
	There are several different proofs of \eqref{general Psi_F Mellin transform}, which makes it of particular interest. A classical approach is to replace the second ${}_1F_1$ function with its Mellin-Barnes integral representation and then interchange the order of integration to obtain a single contour integral. The use of Cauchy's residue theorem will do the rest. A more interesting proof requires the method of brackets (see, for example, \cite{Amdeberhan-Espinosa-Gonzalez} and \cite{Gonzalez-Kondrashuk-Moll-Recabarren-2022}). 
\end{remark}

\section*{Acknowledgement}

The research of the second author is supported by National Natural Science Foundation of China (Grant No. 12001095).


\begin{thebibliography}{99}
	\footnotesize
	\setlength{\itemsep}{-2pt}
		
	\bibitem{Amdeberhan-Espinosa-Gonzalez}
	T. Amdeberhan, O. Espinosa, I. Gonzalez, M. Harrison, V.H. Moll, A. Straub, Ramanujan's master theorem, Ramanujan J. 29 (2012) 103--120.
	
	\bibitem{Andrews-Askey-Roy}
	G.E. Andrews, R. Askey, R. Roy,
	Special Functions,
	Cambridge Univ Press, 1999.
	
	\bibitem{Handbook of Mellin transforms}
	Y.A. Brychkov, O.I. Marichev, N.V. Savischenko, Handbook of Mellin transforms. CRC Press, 2018.
	
	\bibitem{Brychkov-Saad 2014}
	Y.A. Brychkov, N. Saad, On some formulas for the Appell function $F_2(a,b,b';c,c';w,z)$, Integral Transforms Spec. Funct. 25(2) (2014) 111--123.
	
	\bibitem{Brychkov-Savischenko-2023}
	Y.A. Brychkov, N.V. Savischenko, 
	On some formulas for the confluent Horn functions $H_{10}^{(c)}(a;c;w,z)$ and $H_{11}^{(c)}(a,c,c';d;w,z)$
	Integral Transforms Spec. Funct. 34(12) (2023) 915--930.
	
	\bibitem{Choi-Hasanov 2011}
	J. Choi, A. Hasanov, 
	Applications of the operator $H(\alpha,\beta)$ to the Humbert double hypergeometric functions, 
	Comput. Math. Appl. 61 (3) (2011) 663--671.
	
	\bibitem{Antonova-Dmytryshyn-Goran-2023}
	T. Antonova, R. Dmytryshyn, V. Goran, 
	On the analytic continuation of Lauricella-Saran hypergeometric function $F_K(a_1,a_2,b_1,b_2;a_1,b_2,c_3;\mathbf{z})$, 
	Mathematics 2023, 11, 4487.
	
	\bibitem{Dmytryshyn-Goran-2024}
	R. Dmytryshyn, V. Goran, 
	On the analytic extension of Lauricella-Saran's hypergeometric function $F_K$ to symmetric domains, 
	Symmetry 2024, 16, 220.
	
	\bibitem{Higher Transcendental Function Vol 1}
	A. Erd\'{e}lyi, W. Magnus, F. Oberhettinger, F. Tricomi,
	Higher Transcendental Functions, Vol. I. 
	McGraw-Hill, New York, 1981.
	
	\bibitem{Garcia-Lopez 2010}
	E. Garcia, J. L. L\'opez, 
	The Appell's function $F_2$ for large values of its variables, 
	Q. Appl. Math. 68 (4) (2010) 701--712.
		
	\bibitem{Gonzalez-Kondrashuk-Moll-Recabarren-2022}
	I. Gonzalez, I. Kondrashuk, V.H. Moll, L.M. Recabarren, 
	Mellin-Barnes integrals and the method of brackets, 
	Eur. Phys. J. C 82, 28 (2022).
	
	\bibitem{Humbert-1992}
	P. Humbert, 
	The confluent hypergeometric functions of two variables, 
	Edinb. R. S. Proc. 41 (1919), 73--96.
	
	\bibitem{Joshi-Arya-1982}
	C.M. Joshi, J.P. Arya, 
	Inequalities for certain confluent hypergeometric functions of two variables, 
	Indian J. Pure Appl. Math. 13 (1982), 491--500.
	
	\bibitem{Lopez 2008}
	J.L. L\'opez, 
	Asymptotic expansions of Mellin convolution integrals,
	SIAM Rev. 50 (2) (2008) 275--293.
	
	\bibitem{Luo-Raina 2021}
	M.-J. Luo, R.K. Raina, On certain results related to the hypergeometric function $F_K$, J. Math. Anal. Appl.
	504 (2) (2021), 125439.
	
	\bibitem{Luo-Xu-Raina-2022}
	M.-J. Luo, M.-H. Xu, R.K. Raina, On certain integrals related to Saran's hypergeometric function $F_K$, Fractal Frac. 6 (3) (2022).
	
	
	\bibitem{Hai-Yakubovich-Book}
	T.H. Nguyen, S. Yakubovich, The Double Mellin-Barnes Type Integrals and Their Applications to Convolution Theory, World Scientific, Singapore, 1992.
	
	\bibitem{Nemes 2013}
	G. Nemes, An explicit formula for the coefficients in Laplace's method. Constr. Approx. 38 (3) (2013) 471--487 .
	
	\bibitem{Olver 1970}
	F.W.J. Olver, Why steepest descents?. SIAM Rev. 12 (2) (1970) 228--247.
	
	\bibitem{NIST Handbook}
	F.W.J. Olver, D.W. Lozier, R.F. Boisvert, C.W. Clark (Eds.), NIST Handbook of Mathematical Functions, Cambridge University Press, New York, 2010.
	
	\bibitem{Saran 1954}
	S. Saran, Hypergeometric functions of three variables, Ganita 5 (1954) 77--91.
	
	
	\bibitem{Saran 1955}
	S. Saran, Transformations of certain hypergeometric functions of three variables, Acta Math. 93 (1955) 293--312.
	
	\bibitem{Saran 1957}
	S. Saran, Integral representations of Laplace type for certain hypergeometric functions of three variables, Riv., Mat. Univ. Parma 8 (1957) 133--143.
	
	\bibitem{Srivastava-Karlsson}
	H.M. Srivastava, P.W. Karlsson, 
	Multiple Gaussian Hypergeometric Series, 
	Ellis Horwood Ltd., Chichester, 1985.
	
	\bibitem{Wagner 1982}
	E. Wagner, Asymptotische Darstellungen der hypergeometrischen Funktionen f\"{u}r gro\ss e Werte eines Parameters, Z. Anal. Anwend. 1(3) (1982) 1--11.
	
	\bibitem{Wald-Henkel-2018}
	S. Wald, M. Henkel, 
	On integral representations and asymptotics of some hypergeometric functions in two variables, 
	Integral Transforms Spec. Funct. 29(2) (2018) 95--112.
\end{thebibliography}
\end{document}